\documentclass{article}

\usepackage[utf8]{inputenc} 
\usepackage[T1]{fontenc} 
\usepackage{amsmath}
\usepackage{xcolor}
\usepackage{amsthm}
\usepackage{setspace}
\usepackage{mathpazo} 
\usepackage{amssymb}
\usepackage{float}
\usepackage{graphicx}
\usepackage{verbatim}
\usepackage[maxbibnames=99]{biblatex}
\usepackage{ytableau}
\usepackage{hyperref}
\usepackage[autostyle=true]{csquotes}

\newtheorem{teo}{Theorem}
\newtheorem*{teo*}{Theorem}

\newtheorem{lemma}[teo]{Lemma}
\newtheorem{conjecture}[teo]{Conjecture}

\newtheorem{prop}[teo]{Proposition}
\theoremstyle{remark}
\newtheorem{rmk}[teo]{Remark}
\theoremstyle{definition}

\theoremstyle{definition}
\newtheorem{mydef}[teo]{Definition}
\newtheorem{esempio1}[teo]{Example}

\begin{document}
\baselineskip 12pt
\ytableausetup
{mathmode, boxsize=1em}
\begin{center}
\textbf{\Large On chromatic symmetric homology\\and planarity of graphs}
\vspace{5mm}
 \\Azzurra Ciliberti\footnote{La Sapienza Università di Roma - azzurra.ciliberti@uniroma1.it} and Luca Moci\footnote{Università di Bologna - luca.moci2@unibo.it} 
 
 \end{center}
 
 \bigskip

\begin{abstract}
  \noindent In  \cite{1} the authors defined a categorification of Stanley's chromatic symmetric function called the chromatic symmetric homology, given by a suitable family of representations of the symmetric group. In this paper we prove that, as conjectured in \cite{2}, if a graph $G$ is non-planar, then its chromatic symmetric homology in bidegree (1,0) contains $\mathbb{Z}_2$-torsion. Our proof follows a recursive argument based on Kuratowsky's theorem.
\end{abstract}

\section{Introduction}
The \emph{chromatic symmetric function} of a graph, defined by Stanley in \cite{3}, is a remarkable combinatorial invariant which refines the chromatic polynomial. Recenty, in \cite{1}, Sazdanovic and Yip categorified this invariant by defining a new homological theory, called the \emph{chromatic symmetric homology} of a graph $G$. This construction, inspired by Khovanov's categorification of the Jones polynomial \cite{5}, is obtained by assigning a graded representation of the symmetric group to every subgraph of $G$, and a differential to every cover relation in the Boolean poset of subgraphs of $G$. The chromatic symmetric homology is then defined as the homology of this chain complex; its bigraded Frobenius series $Frob_G(q,t)$, when evaluated at $q = t = 1$, reduces
to Stanley's chromatic symmetric function expressed in the Schur basis.

As proved in \cite{2}, this categorification produces a truly stronger invariant: in other words, chromatic symmetric homology can distinguish couples of graphs that have the same chromatic symmetric function. Furthermore, in the same paper, the properties of chromatic symmetric homology with integer coefficients have been investigated. The authors of \cite{2} provided examples of graphs whose chromatic symmetric homology has torsion, leaving open the following conjecture:
\begin{conjecture}
A graph $G$ is non-planar if and only if its chromatic symmetric homology in bidegree (1,0) contains $\mathbb{Z}_2$-torsion.
\end{conjecture}
In this paper we prove one direction of this conjecture, namely:

\begin{teo*}
Let $G$ be a finite non-planar graph. Then its chromatic symmetric homology in bidegree (1,0) contains $\mathbb{Z}_2$-torsion.
\end{teo*}

Our strategy is based on applying Kuratowsky's theorem: we show that the torsion elements in the homology of the complete graph $K_5$ and of the complete bipartite graph $K_{3,3}$  are mapped to torsion elements in the homology of the graphs that are obtained from them by the operations of edge subdivision and graph inclusion, i.e. all the non-planar graphs. 

\section{Computing q-degree zero homology} \label{form}
The interested reader can find a complete description of the construction of chromatic symmetric homology for a graph in the paper \cite{1}. Here we limit ourselves to briefly recall how to compute homology in $q$-degree zero.
For simplifying the notation, we will denote by $C_{i}$ the $i$-th chain module and by $H_{i}$ the $i$-th homology module. They correspond respectively to $C_{i,0}$ and $H_{i,0}$ in the notation of \cite{1}.

Let $G$ be a graph with $n$ vertices and $m$ edges. We can assume without loss of generality that $G$ is simple: indeed if $G$ has a loop, then its chromatic symmetric homology is zero (Proposition 3.1 of \cite{1}), while if $G$ has two vertices connected by multiple edges, then we can replace them by a single edge without affecting the chromatic symmetric homology (Proposition 3.2 of \cite{1}). Hence we denote by $(i,j)$ the edge incident to the vertices $i$ and $j$, and we order the set of edges $E(G)$ lexicographically. 

Each subset $S$ of $E(G)$ is naturally identified with a subgraph of $G$, having the same vertices as $G$ and $S$ as set of edges. The authors of \cite{2,4} call it a "spanning subgraph", while we will call it simply a \emph{subgraph}, because the word "spanning" is sometimes used with a different meaning in graph theory and matroid theory. The set of all subgraphs $G$ has a stucture of Boolean lattice $\bold{B}(G)$, ordered by reverse inclusion. In the Hasse diagram of $\bold{B}(G)$, we direct an edge $\epsilon (F,F')$ from a subgraph $F$ to a  subgraph $F'$ if and only if $F'$ can be obtained by removing an edge $e$ from $F$. The sign of $\epsilon$, $sgn(\epsilon)$, is defined as $(-1)^k$, where $k$ is the number of edges of $F$ less than $e$.

Let $F \subseteq E(G)$ a subgraph of $G$ with
connected components $B_1,\dots,B_r$. Then the module associated to it in $q$-degree zero is the $permutation$ $module$
\begin{center}
    $M_F = Ind_{\frak S_{B_1} \times \dots \times \frak S_{B_r}}^{\frak S_n}(\bold{S}_{(b_1)} \otimes \dots \otimes \bold{S}_{(b_r)}  )$,
\end{center}
where $\frak S_n$ is the permutation group on $n$ elements and $\bold{S}_{(i)}$ is the Specht module related to the partition $(i)$.

We define 
\begin{center}
    $C_i(G)= \displaystyle\bigoplus_{|F|=i} M_F$,
\end{center}
where the sum is over the subgraphs of $G$ with $i$ edges. Therefore the $i$-th chain module $C_i(G)$ of the graph is a direct sum of
${m}\choose{i}$
permutation modules of
$\frak S_n$. If $\lambda$ is the partition whose parts are the sizes of the connected components of $F$, then
$$M_F\cong M_{\lambda}=\mathbb{C}[\frak S_n]\otimes_{\mathbb{C}[S_\lambda]}\bold{S}_{(n)}.$$

Let $F$ and $F'$ be subgraphs of $G$ where $F'=F-e$. There is an $edge$ $map$ $d_{\epsilon(F,F')}:M_F \to M_{F'}$, defined in our case as the inclusion (for the general definition see \cite{1}).

Finally, the $i$-$th$ $chain$ $map$ $d_i : C_i(G) \to C_{i-1}(G)$ is defined as
\begin{center}
    $d_i=\displaystyle\sum_\epsilon sgn(\epsilon) d_\epsilon$,
\end{center}
where the sum is over all the edges $\epsilon$ in $\bold{B}(G)$ which join a subgraph of $G$ with $i$ edges to a subgraph with $i-1$ edges. Sometimes we will use the notation $d_i^G$, where it may not be clear which graph we are referring to.

We need to recall the following definitions from \cite{2}.
\begin{mydef}
Let $F$ be a subgraph of $G$, and let $\lambda \vdash n$ be the partition whose
parts are the sizes of the connected components of $F$. The numbering $T(F)$ associated to
$F$ is the numbering of shape $\lambda$ such that each row consists of the elements in a connected
component of $F$ arranged in increasing order, and rows of $T(F)$ having the same size are
ordered so that the minimum element in each row is increasing down the first column.
\end{mydef}
Let $T = T(F)$. The $q$-$degree$ $zero$ $permutation$ $module$ $M_T$ associated to the numbering $T$ 
is cyclically generated by the Young symmetrizer $$a_T=\displaystyle\sum_{\rho \in R(T)}\rho$$ where $R(T)\leq \frak S_n$ is the subgroup of permutations that permute elements within each row of $T$. We have:
\begin{center}
    $M_F \cong  M_T = \mathbb{C}[\frak S_n] \cdot a_T$.
\end{center}

\begin{mydef}
For any numberings $S$ and $T$ of shape $\lambda$, let
\begin{center}
    $v_T^S= \sigma_{T,S}b_Ta_T=b_Sa_S\sigma_{T,S} \in M_T$,
\end{center}
where $a_T$ is as above, $$b_T=\displaystyle\sum_{\zeta \in C(T)} sgn(\zeta)\zeta,$$ $C(T)\leq \frak S_n$ is the subgroup of permutations that permute elements within each column of $T$, and $\sigma_{T,S}\in \frak S_n$ is such that $\sigma_{T,S}\cdot T = S$.
\end{mydef}
Moreover, the $q$-$degree$ $zero$ $Specht$ $module$ $\bold{S}_T$ associated to the numbering $T$ is cyclically generated by the Young symmetrizer $c_T=b_Ta_T$:
\begin{center}
    $\bold{S}_T=\mathbb{C}[\frak S_n]\cdot c_T$,
\end{center}
and $\bold{S}_T \cong \bold{S}_\lambda$.

We also recall the following result from \cite{4}, Section 7.2, Proposition 2.
\begin{prop}
 Let $F$ be a subgraph of $G$ with associated numbering $T$ of shape $\lambda$. Then
 \begin{center}
     $\bold{S}_T=span\{b_Sa_S\sigma_{T,S} | S \in SYT(\lambda)\} = span \{v_T^S | S \in SYT(\lambda)\}$,
 \end{center}
 where $SYT(\lambda)$ is the set of standard Young tableaux of shape $\lambda$.
\end{prop}

\subsection{Computation of $H_1(G)$}\label{1.1}
We will describe the chromatic homology
in terms of Specht modules. Since each Specht module is cyclically generated, then our inclusion
maps are completely determined by specifying the image of a cyclic generator for each Specht
module. We now show how to achieve these computations systematically.\\
We restrict ourselves to Specht
modules of type $\lambda = (2^k, 1^{n-2k)}$ for $k\geq1$, so
we will be computing
\begin{center}
    $C_2(G)_{|_{S_\lambda}}\xrightarrow{d_2}C_1(G)_{|_{S_\lambda}}\xrightarrow{d_1}C_0(G)_{|_{S_\lambda}}\rightarrow 0.$
\end{center}
We order the edges of $G$ in lexicographic order and label these as $e_1,\dots, e_{m}$.
In homological degree zero, there is only one subgraph without edges. The chain
group $C_0(G) = M_{F_\emptyset}\cong M_{(1^n)}$ is the regular representation of $\frak S_n$, where $F_\emptyset$ is the edgeless
subgraph. By Corollary 1 in Section 7.3 of \cite{4}, the multiplicity of $S_\lambda$ in $\mathbb{C}[S_n]$ is the number $f^\lambda$ = $K_{\lambda,(1^n)}$ of standard Young
tableaux of shape $\lambda$. We list the tableaux $Y_1(G),\dots,Y_{f^\lambda}(G)\in SYT(\lambda)$ with respect to the following total
order:  if $T$ and $S$ are numberings of shape $\lambda$ such that the $i$-th row is the lowest row in which the numberings are different, the $j$-th column is the rightmost column in that row in which the numberings are different and $T(i,j)>S(i,j)$, then we say that $T>S$.\\
We have
\begin{center}
     $C_0(G)|_{S_\lambda} = \displaystyle\bigoplus_{i=1}^{f^\lambda}\mathbb{C}[\frak S_n]\cdot v_{Y_i}^{Y_1}$.
\end{center}

\begin{esempio1}
Let $G=K_5$. We order the edges of $G$ in lexicographic order; that is,
\begin{center}
    $(1,2),(1,3),(1,4),(1,5),(2,3),(2,4),(2,5),(3,4),(3,5),(4,5)$,
\end{center}
and label these as $e_1,\dots, e_{10}$. The standard Young tableaux of shape $\lambda = (2^2,1)$ listed with respect to the ordering defined earlier are
\begin{center}
$Y_1=$ \begin{ytableau}
    1 & 2 \\ 3 & 4 \\ 5
    \end{ytableau} $Y_2=$ \begin{ytableau}
    1 & 2 \\ 3 & 5 \\ 4
    \end{ytableau} $Y_3=$ \begin{ytableau}
    1 & 3 \\ 2 & 4 \\ 5
    \end{ytableau} $Y_4=$ \begin{ytableau}
    1 & 3 \\ 2 & 5 \\ 4
    \end{ytableau} $Y_5=$ \begin{ytableau}
    1 & 4 \\ 2 & 5 \\ 3
    \end{ytableau}
\end{center}

Then
\begin{center}
    $C_0(K_5)_{|_{\bold{S}_{(2^2,1)}}}= \displaystyle\bigoplus_{i=1}^5 (\mathbb{C}[S_5]\cdot v_{Y_i}^{Y_1}) \cong \bold{S}_{(2^2,1)}^{\oplus 5}$.
\end{center}
\end{esempio1}

In homological degree one, there are $m$ subgraphs with exactly one
edge, thus $$C_1(G) = \displaystyle\bigoplus_{i=1}^{m} M_{F_{e_i}}.$$
If $F_{e_ i}$ is the subgraph containing the edge $e_i = (p,q)$, then the permutation module $M_{F_{e_i}}$
has the associated numbering $T(F_{e_i})$ of shape
$\mu = (2, 1^{n - 2}),$
and $M_{F_{e_i}}= \mathbb{C} [\frak S_n]\cdot (e + (pq))\cong M_\mu$.\\ 
The multiplicity of $S_\lambda$ in $M_{F_{e_i}}$
is the number $K_{\lambda,\mu}$ of semistandard Young tableaux of
shape $\lambda$ and weight $\mu$. We next obtain numberings of shape $\lambda$ that will index these $K_{\lambda,\mu}$
Specht modules $\bold{S}_\lambda$, by standardizing the set $SSYT(\lambda, \mu)$ of semistandard Young tableaux of
shape $\lambda$ and weight $\mu$ with respect to $T(F_{e_i})$ in the following way. For any numbering $S$, the
word $w(S)$ of $S$ is obtained by reading the entries of the rows of $S$ from left to right, and
from the top row to the bottom row (note that this is not the usual definition of a reading
word for tableaux). So, given $Y \in SSYT(\lambda, \mu)$, let $w(Y ) = y_1,\dots,y_n$ be the word of $Y$, let $w(T) = t_1,\dots,t_n$ be the word of $T = T(F_{e_i})$ and let $\sigma$ be the permutation that orders $y_1,\dots,y_n$ without exchanging $y_i$ and $y_j$ if $y_i=y_j$. From this we obtain a numbering $X$ of
shape $\lambda$ by replacing the entry in $Y$ that corresponds to $y_k$ by $t_{\sigma(k)}$ . We list the numberings $X_i^1(G),\dots,X_i^{K_{\lambda,\mu}}(G)$
obtained using the procedure just described to $SSYT(\lambda,\mu)$ with respect to $T(F_{e_i} )$.
Observe that since $\mu = (2,1^{n-2})$ and $\lambda = (2^k, 1^{n-2k})$ where $k\geq 1$, then this procedure guarantees that the first row of each numbering $X_i^j(G)$ is 
\begin{ytableau}
p & q
\end{ytableau}
. So $v_{X_i^j(G)}^{Y_1} \in M_{F_{e_i}}$ and $\mathbb{C}[\frak S_n]\cdot v_{X_i^j(G)}^{Y_1} \cong$ $\bold{S}_\lambda$ for $j=1,\dots,K_{\lambda,\mu}$.
Thus
\begin{center}
    $C_1(G)_{|_{S_\lambda}} = \displaystyle\bigoplus_{i=1}^{m} \displaystyle\bigoplus_{j=1}^{K_{\lambda,\mu}} \mathbb{C}[S_n]\cdot v_{X_i^j(G)}^{Y_1}$
\end{center}
\begin{esempio1}
Let $G=K_5$. There are 10 subgraphs with exactly one edge. Furthermore, there are two semistandard Young tableaux of shape $\lambda=(2^2,1)$ and weight $(2,1^3)$,
\begin{center}
    $Z_1 = $\begin{ytableau}
    1 & 1 \\ 2 & 3 \\ 4
    \end{ytableau} \hspace{0.6cm}and \hspace{0.6cm} $Z_2 = $\begin{ytableau}
    1 & 1 \\ 2 & 4 \\ 3
    \end{ytableau} \hspace{0,1cm},
\end{center}
so the multiplicity of $\bold{S}_{(2^2,1)}$ in each $M_{F_{e_i}}\cong M_{(2,1^3)}$ is 2. Let $X_i^1(K_5)$
and $X_i^2(K_5)$ denote the numberings which index the two copies of $\bold{S}_{(2^2,1)}$ in each $M_{F_{e_i}}$, again listed with respect to the same ordering. So
\begin{center}
    $C_1(K_5)|_{\bold{S}_{(2^2,1)}}= \displaystyle\bigoplus_{i=1}^{10} (\mathbb{C}[S_5]\cdot v_{X_i^1(K_5)}^{Y_1} \oplus \mathbb{C}[S_5]\cdot v_{X_i^2(K_5)}^{Y_1}) \cong \bold{S}_{(2^2,1)}^{\oplus 20}$.
\end{center}
Consider for instance the subgraph $F_{e_1}$ of $K_5$ with the edge $e_1 = (1,2)$ only. The numbering associated to it is
\begin{center}
    $T(F_{e_1}) = \begin{ytableau}
    1 & 2 \\ 3 \\ 4 \\ 5
    \end{ytableau}$.
\end{center}
Let $z_1^i,\dots, z_5^i$ be the word of $Z_i$, $i=1,2$. The permutation that orders $z_1^1,\dots,z_5^1$ is the identity and the one that orders $z_1^2,\dots,z_5^2$ is $(45)$. Therefore we have 
\begin{center}
    $X_1^1 = \begin{ytableau}
    1 & 2 \\ 3 & 4 \\ 5
    \end{ytableau} = Y_1$\hspace{0.5cm} and \hspace{0.5cm}$X_1^2 = \begin{ytableau}
    1 & 2 \\ 3 & 5 \\ 4
    \end{ytableau} = Y_2$,
\end{center}
then
\begin{center}
    $v_{X_1^1}^{Y_1}=v_{Y_1}^{Y_1}$ \hspace{0.5cm}and\hspace{0.5cm} $v_{X_1^2}^{Y_1}=v_{Y_2}^{Y_1}$. 
\end{center}
\end{esempio1}
Lastly, we consider the chain module in homological degree two. The subgraphs of $G$
with exactly two edges have connected components of partition type $(2^2, 1^{n-4})$ or $(3,1^{n-3})$.
We are only concerned with Specht modules of type $\lambda = (2^k, 1^{n-2k})$ with $k\geq2$ necessarily
and, since $\lambda \ntriangleright (3, 1^{n-3})$, then, by Corollary 1 in Section 7 of \cite{4}, $S_\lambda$ does not appear as a summand in a permutation module
isomorphic to $M_{(3,1^{n-3})}$. Hence, we only need to consider the subgraphs with
connected components of partition type $(2^2,1^{n-4})$.\\
So suppose $G$ has $h$ subgraphs whose connected components has partition type
$\nu = (2^2, 1^{n-4})$. List these subgraphs with respect to the lexicographic order of
their edge sets. Suppose $F_{e_i,e_j}$ is the subgraph that contains the edges
$e_i$ = $(p,q)$ and $e_j = (r,s)$ with $p < r$. The permutation module $M_{F_{e_i,e_j}}$
has the associated
numbering $T(F_{e_i,e_j})$ of shape $\nu$ and $$M_{F_{e_i,e_j}} = \mathbb{C}[\frak S_n]\cdot (e + (pq))(e + (rs))\cong M_\nu.$$

Similar to the previous case for $C_1(G)$, the multiplicity of $S_\lambda$ in $M_{F_{e_i,e_j}}$ is $K_{\lambda,\nu}$. We
list the numberings $W_{i,j}^1(G),\dots,W_{i,j}^{K_{\lambda,\nu}}(G)$ obtained using the procedure described above to $SSYT(\lambda,\nu)$ with respect to
$T(F_{e_i,e_j})$. The procedure guarantees that
the top two rows of each numbering $W$ are
\begin{ytableau}
p & q\\
r & s
\end{ytableau}
, so $v_W^{Y_1}\in M_{F_{e_i,e_j}}$
and $$\mathbb{C}[S_n]\cdot v_{W_{i,j}^l(G)}^{Y_1}\cong \bold{S}_\lambda\text{ for }l=1,\dots,K_{\lambda,\nu}.$$ Thus
\begin{center}
     $C_2(G)_{|_{S_\lambda}} = \displaystyle\bigoplus_{i,j} \displaystyle\bigoplus_{l=1}^{K_{\lambda,\nu}} \mathbb{C}[S_n]\cdot v_{W_{i,j}^l(G)}^{Y_1}$,
\end{center}
where the direct sum is over the values of $i$ and $j$ corresponding to the couples of
edges which form subgraphs of type $\nu$, i.e. the non-consecutive edges.
\begin{esempio1}
Let $G=K_5$. There are 15 subgraphs whose connected components have partition type $(2^2,1)$. There is only one semistandard Young tableau of shape $\lambda=(2^2,1)$ and weight $(2^2,1)$:
\begin{center}
    \begin{ytableau}
    1 & 1 \\ 2 & 2 \\ 3
    \end{ytableau}
\end{center}
so the multiplicity of $\bold{S}_{(2^2,1)}$ in each $M_{F_{e_i,e_j}}\cong M_{(2^2,1)}$ is 1, and we let $W_{i,j}$ denote the numbering which indexes the copy of $\bold{S}_{(2^2,1)}$ in $M_{F_{e_i,e_j}}$. Therefore,
\begin{center}
    $C_2(K_5)_{|_{\bold{S}_{(2^2,1)}}}= \displaystyle\bigoplus (\mathbb{C}[S_5]\cdot v_{W_{i,j}}^{Y_1}) \cong \bold{S}_{(2^2,1)}^{\oplus 15}$,
\end{center}
where the direct sum is over the values of $i$ and $j$ corresponding to the couples of non-consecutive edges.
\end{esempio1}

To compute the edge maps we will need the following theorem (Corollary 2.18 of \cite{2}).
\begin{teo}\label{pij}
For any numberings $S$ of shape $\lambda$,
\begin{center}
  $v_S^T = (-1)^j \displaystyle\sum_{U\in \Xi_{i,j}(S)} v_U^T$,  
\end{center}
where $\Xi_{i,j}(S)$ is the set of all numberings $U$ obtained from $S$ by exchanging the first $j$ entries in the $(i+1)$-th row of $S$ with $j$ entries in the $i$-th row of $S$, preserving the order of each subset of elements.

\end{teo}
We let $\pi_{i,j}$ denote the operator on numberings such that
\begin{center}
    $\pi_{i,j}(S)=(-1)^j\displaystyle\sum_{U\in \Xi_{i,j}(S)}U$.
\end{center}

\begin{esempio1}
Let $G=K_5$. We compute $d_1(v_{X_3^1}^{Y_1})$ and $d_1(v_{X_3^2}^{Y_1})$, so we consider the subgraph $F_{e_3}$ of $K_5$ with the edge $e_3 = (1,4)$ only. The numbering associated to it is
\begin{center}
    $T(F_{e_3}) = \begin{ytableau}
    1 & 4 \\ 2 \\ 3 \\ 5
    \end{ytableau}$.
\end{center}
We have 
\begin{center}
    $X_3^1 = \begin{ytableau}
    1 & 4 \\ 2 & 3 \\ 5
    \end{ytableau}$\hspace{0.5cm} and \hspace{0.5cm}$X_3^2 = \begin{ytableau}
    1 & 4 \\ 2 & 5 \\ 3
    \end{ytableau} = Y_5$.
\end{center}
We compute 
\begin{center}
    $\pi_{1,1}(X_3^1)= - \begin{ytableau}
    1 & 2 \\ 4 & 3 \\ 5
    \end{ytableau} - \begin{ytableau}
    2 & 4 \\ 1 & 3 \\ 5
    \end{ytableau}$
\end{center}
and 
\begin{center}
    $\pi_{1,2} \begin{ytableau}
    2 & 4 \\ 1 & 3 \\ 5
    \end{ytableau} = \begin{ytableau}
    1 & 3 \\ 2 & 4 \\ 5
    \end{ytableau} = Y_3$.
\end{center}
Then, by Theorem \ref{pij}, we have
\begin{center}
    $v_{X_3^1}^{Y_1}= - v_{Y_1}^{Y_1} - v_{Y_3}^{Y_1}$ \hspace{0.5cm}and\hspace{0.5cm} $v_{X_3^2}^{Y_1}=v_{Y_5}^{Y_1}$,
\end{center}
and $d_1$ sends
\begin{center}
    $v_{X_3^1}^{Y_1} \mapsto - v_{Y_1}^{Y_1} - v_{Y_3}^{Y_1}$ \hspace{0.5cm}and\hspace{0.5cm} $v_{X_3^2}^{Y_1} \mapsto v_{Y_5}^{Y_1}$.
\end{center}
Now we compute $d_2(v_{W_{1,8}}^{Y_1})$, so we have to consider the subgraph $F_{e_1,e_8}$ of $K_5$ with the edges $e_1=(1,2)$ and $e_8=(3,4)$. The numbering associated to it is
\begin{center}
    \begin{ytableau}
    1 & 2 \\ 3 & 4 \\ 5
    \end{ytableau}.
\end{center}
There are two subgraphs of $K_5$ with only one edge that can be obtained removing one edge from $F_{e_1,e_8}$, i.e. $F_{e_1}$ and $F_{e_8}$. The per-edge map $d_{\epsilon(F_{e_1,e_8},F_{e_1})}$ appears as a summand in $d_2$ with a minus sign; instead $d_{\epsilon(F_{e_1,e_8},F_{e_8})}$ appears in $d_2$ with a plus sign. Therefore, $d_2$ sends
\begin{center}
    $v_{W_{1,8}}^{Y_1} \mapsto - v_{X_1^1}^{Y_1} + v_{X_8^1}^{Y_1}$.
\end{center}
\end{esempio1}

\section{The case of non-planar graphs} 
In this section we will prove that if $G$ is a non-planar graph, then the chromatic symmetric homology $H_1(G;\mathbb{Z})$ contains $\mathbb{Z}_2$-torsion. We first recall two results from \cite{2}:
\begin{lemma}\label{lemma1}
The chromatic symmetric homology $H_1(K_5;\mathbb{Z})$ contains $\mathbb{Z}_2$-torsion.
\end{lemma}
\begin{proof}
We compute 
\begin{center}
     $C_2(K_5)_{|_{\bold{S}_{(2^2,1)}}}\xrightarrow{d_2}C_1(K_5)_{|_{\bold{S}_{(2^2,1)}}}\xrightarrow{d_1}C_0(K_5)_{|_{\bold{S}_{(2^2,1)}}}\rightarrow 0$,
\end{center}
restricted to the $\bold{S}_{(2^2,1)}$ modules.

Following the notation introduced in \ref{1.1}, let
$g = W_{1,8} + W_{1,9} + W_{1,10} + W_{2,6} - W_{2,7} - W_{2,10} + W_{3,5}
+ W_{3,7} + W_{3,9} + W_{4,5} + W_{4,6} + W_{4,8} - W_{5,10} - W_{6,9} + W_{7,8} \in C_2(K_5)$ and $h = X_9^1 + X_{10}^1 - X_2^1 + X_7^2 + X_9^2 \in C_1(K_5)$. We have that $h \notin \mathrm{im}\,d_2$, $d_2(g) = 2h$ and $d_1(h) = 0$, so $h$ generates $\mathbb{Z}_2$-torsion in $H_1(K_5;\mathbb{Z})$. For more details see \cite{2}, Theorem 4.1.

\end{proof}
\begin{lemma}\label{lemma2}
The chromatic symmetric homology $H_1(K_{3,3};\mathbb{Z})$ contains $\mathbb{Z}_2$-torsion.
\end{lemma}
\begin{proof}
We compute 
\begin{center}
     $C_2(K_{3,3})|_{\bold{S}_{(2^2,1^2)}}\xrightarrow{d_2}C_1(K_{3,3})|_{\bold{S}_{(2^2,1^2)}}\xrightarrow{d_1}C_0(K_{3,3})|_{\bold{S}_{(2^2,1^2)}}\rightarrow 0$,
\end{center}
restricted to the $\bold{S}_{(2^2,1^2)}$ modules.

Following the notation introduced in Section \ref{1.1}, let
$g' = W_{1,6} - W_{1,7} + W_{1,8} + W_{1,9} - W_{2,4} - W_{2,5} + W_{2,7}
+ W_{2,9} + W_{3,4} - W_{3,5} + W_{3,6} + W_{3,8} + W_{4,8} + W_{4,9} + W_{5,6} + W_{5,7} - W_{6,9} + W_{7,8} \in C_2(K_{3,3})$ and $h' = X_6^3 - X_{7}^3 + X_8^3 - X_9^2 \in C_1(K_{3,3})$. We have that $h' \notin \mathrm{im}\,d_2$, $d_2(g') = 2h'$ and $d_1(h') = 0$, so $h'$ generates $\mathbb{Z}_2$-torsion in $H_1(K_{3,3};\mathbb{Z})$. For more details see \cite{2}, Theorem 4.2.

\end{proof}

\begin{prop}\label{p1}
 Let $G$ be a graph with $n$ vertices, $n \geq 2$, $\lambda$ a partition of $n$ of type $(2^k,1^{n-2k})$ and $h \in C_1(G)_{|\bold{S}_\lambda}$ a generator of $\mathbb{Z}_p$-torsion in $H_1(G;\mathbb{Z})$. Let $G'$ be a subdivision of $G$ with $n'$ vertices, $n'> n$, i.e. a graph obtained from $G$ by inserting $n' - n$ vertices into the edges of $G$. Then there exists $h' \in C_1(G')_{|\bold{S}_{\lambda'}}$, with $\lambda'=(2^k, 1^{n'-2k})$, that generates $\mathbb{Z}_p$-torsion in $H_1(G',\mathbb{Z})$. 
\end{prop}

\begin{proof}
it is enough to show that the statement holds for $n'=n+1$. 

Let $G$ be a graph with $n$ vertices and $G'$ be the graph obtained from $G$ by inserting a vertex into an edge of $G$. By hypothesis, there exists $h \in C_1(G)_{|\bold{S}_\lambda}$, with $\lambda = (2^k,1^{n-2k})$, that generates $\mathbb{Z}_p$-torsion in $H_1(G;\mathbb{Z})$.

We number the vertices of $G'$
from 1 to $n+1$, so that the vertex added is $n+1$.

We prove that $h$ is mapped in a $\mathbb{Z}_p$-torsion generator in $H_1(G';\mathbb{Z})$.

We have that all the edges of $G$ are also edges of $G'$, except for the edge that has been broken, let it be $(i,j)$, which is no longer an edge of $G'$, but it has been replaced by two edges, $(i,n+1)$ and $(n+1,j)$. We consider the partition $\lambda'=(2^k, 1^{n+1-2k})$, i.e. the partition $\lambda$ with an extra box at the bottom. We divide the standard Young tableaux of shape $\lambda'$ into two groups: those that are obtained simply by adding the box containing 
$n+1$ at the bottom of the standard Young tableaux of shape $\lambda$ and those that don't have $n+1$ in the last row. We do the same for the semistandard Young tableaux of shape $\lambda'$ and weight $(2, 1^{n-1})$. 
 If $i$ does not correspond to the two new edges, the $X_i^l(G')$'s are identical to the $X_i^{l}(G)$'s with the box containing $n+1$ added at the bottom. Therefore, the differential $d_1^{G'}$ acts on them in exactly the same way as $d_1^{G}$, since the $\pi_{i,j}$ operations don't concern the last row. Since $h\in \mathrm{ker}\, d_1^{G}$, we also have that $h\in \mathrm{ker}\,d_1^{G'}$.

We mean $h$ written as 1-chain in $G'$. This can always be done using the $\pi_{i,j}$ operations and Theorem \ref{pij}.

For example, consider $K_5$ and the graph $G_1$ obtained by adding the vertex 6 into the edge $(1,5)$. 
$X_4^1(K_5)$ with a box containing 6 added at the bottom, i.e.
\begin{ytableau}
1 & 5\\
2 & 3\\
4\\
6
\end{ytableau}, a priori is not a 1-chain in $G_1$ since the edge (1,5) is not an edge of $G_1$; but we have
\begin{center}
  \begin{ytableau}
1 & 5\\
2 & 3\\
4\\
6
\end{ytableau} = $\pi_{1,1} \begin{ytableau}
1 & 5\\
2 & 3\\
4\\
6
\end{ytableau} = -  \begin{ytableau}
2 & 5\\
1 & 3\\
4\\
6
\end{ytableau} -  \begin{ytableau}
1 & 2\\
3 & 5\\
4\\
6
\end{ytableau} = -X_2^2(G_1) - X_1^2(G_1) \in C_1(G_1)$.  
\end{center}
\begin{rmk}\label{rmk}
It cannot happen that there is cancellation among 1-cycles in $G'$ that give the cycle $h$ in $G$. In fact, the $\pi_{i,j}$ operations on numberings correspond, by Theorem \ref{pij}, to equalities among the elements of $\bold{S}_{(2^k,1^{n-2k})}$ indexed by such numberings. Therefore, if there was cancellation among 1-cycles in $G'$ that give $h$ in $G$, it would also be among the 1-cycles of $G$ that form $h$, so $h$ would be the trivial 1-cycle of $G$, which is not true.
\end{rmk}
A similar argument applies to the $W_{h,k}^s$'s. Since, in $G$, there exists a 2-cycle $g$ such that $d_2^{G}(g) = 2h$ and $g$ can be written as 2-chain in $G'$ with the $\pi_{i,j}$ operations, without becoming trivial as observed for $h$ in Remark \ref{rmk}, we also have that $d_2^{G'}(g) = 2h$.\\
It remains to prove that $h \notin \mathrm{im}\,d_2^{G'}$. If $h$ belonged to $\mathrm{im}\,d_2^{G'}$, since we know that $h \notin \mathrm{im}\,d_2^{G}$, it would be linear combination of the columns of $d_2^{G'}$ which are not in $d_2^{G}$, but it is not possible because $h$ is a 1-chain in $G$.\\
Therefore, we have a $\mathbb{Z}_p$-torsion generator in $H_1(G',\mathbb{Z})$.
\end{proof}

\begin{prop}\label{p2}
Let $G'$ be a graph with $n'$ vertices and let $G$ be a subgraph of $G'$ with $n$ vertices, $n \leq n'$. Assume that $h \in C_1(G)_{|\bold{S}_\lambda}$ is a generator of $\mathbb{Z}_p$-torsion in $H_1(G;\mathbb{Z})$. Then there exists $h' \in C_1(G')_{|\bold{S}_{\lambda'}}$, with $\lambda'=(2^k, 1^{n'-2k})$, that generates $\mathbb{Z}_p$-torsion in $H_1(G',\mathbb{Z})$.  
\end{prop}

\begin{proof}
We have that all the edges of $G$ are also edges of $G'$. We consider the partition $\lambda'=(2^k, 1^{n'-2k})$, i.e. the partition $\lambda$ with $n'-n$ extra boxes at the bottom. We number the vertices of $G'$
from 1 to $n'$, so that the vertices eventually added are $n+1, \dots, n'$. We divide the standard Young tableaux of shape $\lambda'$ into two groups: those that are obtained simply by adding the boxes containing 
$n+1, \dots, n'$ at the bottom of the standard Young tableaux of shape $\lambda$ and those that don't have $n+1, \dots, n'$ in the last rows. We do the same for the semistandard Young tableaux of shape $\lambda'$ and weight $(2, 1^{n'-2})$. All the $X_i^{l}(G)$'s with extra boxes containing $n+1, \dots, n'$ at the bottom are among the $X_i^{l'}(G')$'s. Therefore, the differential $d_1^{G'}$ acts on them in exactly the same way as $d_1^{G}$, since the $\pi_{i,j}$ operations don't concern the last rows. Since $h\in \mathrm{ker}\, d_1^{G}$, we also have that $h\in \mathrm{ker}\,d_1^{G'}$.
 
 A similar argument applies to the $W_{h,k}^s$'s. Since, in $G$, there exists a 2-cycle $g$ such that $d_2^{G}(g) = 2h$, we also have that $d_2^{G'}(g) = 2h$.\\
It remains to prove that $h \notin \mathrm{im}\,d_2^{G'}$. If $h$ belonged to $\mathrm{im}\,d_2^{G'}$, since we know that $h \notin \mathrm{im}\,d_2^{G}$, it would be linear combination of the columns of $d_2^{G'}$ which are not in $d_2^{G}$, but it is not possible because $h$ is a 1-chain in $G$.\\
Therefore, we have a $\mathbb{Z}_p$-torsion generator in $H_1(G',\mathbb{Z})$.
\end{proof}

\begin{teo}\label{mail}
Let $G$ be a finite non-planar graph. Then $H_1(G;\mathbb{Z})$ contains $\mathbb{Z}_2$-torsion.
\end{teo}

\begin{proof}
Since $G$ is non-planar, by Kuratowsky's theorem it 
contains a subgraph $G'$ which is a subdivision of
 $K_5$ or $K_{3,3}$. By Lemma \ref{lemma1} and Lemma \ref{lemma2}, both $H_1(K_5;\mathbb{Z})$ and $H_1(K_{3,3};\mathbb{Z})$ have a generator of $\mathbb{Z}_2$-torsion of type $(2^k,1^{n-2k})$. Hence, by Proposition \ref{p1}, also $H_1(G';\mathbb{Z})$ contains such a  $\mathbb{Z}_2$-torsion element; thus by Proposition \ref{p2} also $H_1(G;\mathbb{Z})$ does.  
\end{proof}

\printbibliography[heading=bibintoc]





\end{document}